\documentclass[final,leqno,onefignum,onetabnum]{siamltex1213}
\usepackage{amsmath}
\usepackage{amsfonts}
\usepackage{color}

\newtheorem{rem}[theorem]{Remark}

\title{Unique recovery of piecewise analytic density and stiffness tensor
  from the elastic-wave Dirichlet-to-Neumann map}

\author{Maarten V. de Hoop\thanks{Simons Chair in Computational and
    Applied Mathematics and Earth Science, Rice University, Houston,
    TX 77005, USA (\tt{mdehoop@rice.edu}).}
\and
Gen Nakamura\thanks{Department of Mathematics, Hokkaido University,
  Sapporo 060-0810, Japan (\tt{nakamuragenn@gmail.com}).}
\and
Jian Zhai\thanks{Department of Mathematics,
 University of Washington, Seattle, WA 98195, USA
  (\tt{jian.zhai@outlook.com}).}}

\begin{document}
\maketitle
\slugger{siap}{xxxx}{xx}{x}{x--x}

\begin{abstract}
We study the recovery of piecewise analytic density and stiffness
tensor of a three-dimensional domain from the local dynamical
Dirichlet-to-Neumann map. We give global uniqueness results if the
medium is transversely isotropic with known axis of symmetry or
orthorhombic with known symmetry planes on each subdomain. We also
obtain uniqueness of a fully anisotropic stiffness tensor, assuming
that it is piecewise constant and that the interfaces which separate the
subdomains have curved portions. The domain partition need not to be
known. Precisely, we show that a domain partition consisting of
subanalytic sets is simultaneously uniquely determined.
\end{abstract}



\pagestyle{myheadings} \thispagestyle{plain} \markboth{Unique recovery of density and stiffness tensor}{Maarten
  V. de Hoop, Gen Nakamura and Jian Zhai}

\section{Introduction}${}$\newline\indent
We study the recovery of piecewise analytic density and stiffness
tensors of a three-dimensional domain from the local dynamical
Dirichlet-to-Neumann map. We introduce a domain partition and consider
anisotropy and scattering off the interfaces separating the subdomains
in the partition. This has been considered as an open problem in exploration seismology where anisotropy reveals
critical information on earth materials, microstructure in geological
formations, and stress. The stress inducted anisotropy is analyzed in \cite{JR,TE}

We let $\Omega\subset\mathbb{R}^3$ be a bounded domain with smooth
boundary $\partial\Omega$ and $y=(y^1,y^2,y^3)$ be Cartesian coordinates. We consider the following initial boundary
value problem for the system of equations describing elastic waves
\begin{equation}\label{EQ no1}
\begin{cases}
\rho\partial^2_tu=\operatorname{div} (\mathbf{C}\varepsilon(u))=:Lu~~\text{in}~\Omega_T=\Omega\times(0,T) ,\\
u=f~~\text{on}~\partial\Omega\times(0,T) ,\\
u(y,0)=\partial_t u(y,0)=0~~\text{in}~\Omega
\end{cases}
\end{equation}
with $f(y,0)=0$ and $\frac{\partial}{\partial t}f(y,0)=0$ for $y\in\partial\Omega$.
Here, $u$ denotes the displacement vector and
\[\varepsilon(u)=(\nabla u+(\nabla u)^T)/2=(\varepsilon_{ij}(u))=\frac{1}{2}\left(\frac{\partial u_i}{\partial y^j}+\frac{\partial u_j}{\partial y^i}\right)\]
the linear strain tensor which is the symmetric part of $\nabla u$. Furthermore, $\mathbf{C}=(C^{ijkl})=(C^{ijkl}(y))$ is the stiffness
tensor and $\rho=\rho(y)$ is the density of mass, which are piecewise analytic on $\overline\Omega$. 

Here, the hyperbolic or dynamical Dirichlet-to-Neumann map (DN map)
$\Lambda_T$ is given as the mapping
\begin{equation}\label{Lambda_T}
\Lambda_T: f\mapsto\partial_L u:=(\mathbf{C}\varepsilon(u))\nu|_{\partial\Omega},
\end{equation}
where $u$ is the solution of (\ref{EQ no1}),
$\mathbf{C}\varepsilon(u)$ is a $3\times3$ matrix with its $(i,j)$
component $(\mathbf{C}\varepsilon(u))^{ij}$ given by
$(\mathbf{C}\varepsilon(u)) ^{ij}=\sum_{k,l=1}^3
C^{ijkl}\varepsilon_{kl}(u)$, $\nu$ is the outward unit normal to
$\partial\Omega$. Physically, $\partial_L u$ signifies the normal traction at
$\partial\Omega$. Its mapping property, that is, the domain and target
spaces, will be specified in Section~2. Actually, we will consider a local DN map which is a localized version of the DN map. 
We are using the (full) DN map here just for simplicity.

It is physically natural to assume that $\rho$ is bounded away from
$0$ on $\overline{\Omega}$ and that the stiffness tensor $\mathbf{C}$
satisfies the following symmetries and strong convexity condition:

\begin{itemize}
\item (symmetry) $C^{ijkl}(x)=C^{jikl}(x)=C^{klij}(x)$ for any $x\in\overline\Omega$ and $i,j,k,l$; 
\item (strong convexity) there exists a $\delta>0$ such that for any 
$3\times 3$ real-valued symmetric matrix $(\varepsilon_{ij})$,
\[
   \sum_{i,j,k,l=1}^3C^{ijkl}\varepsilon_{ij}\varepsilon_{kl}\geq\delta\sum_{i,j=1}^3\varepsilon_{ij}^2.
\]
\end{itemize}

We first consider the following inverse boundary value problem: Can one
determine $C^{ijkl}$ and $\rho$ (as well as all their derivatives) at
the boundary from $\Lambda_T$? This inverse problem is referred to as the
\textit{boundary determination}.  Concerning the uniqueness, this question was first answered by
Rachelle \cite{Rachelle} for the isotropic case, that is
\[\mathbf{C}=\left(\lambda\delta^{ij}\delta^{kl}+\mu(\delta^{ik}\delta^{jl}+\delta^{il}\delta^{jk})\right).\]
Her method depends on the decoupling of S- and P- waves on the
boundary. This separation of polarizations, however, is not required
for our proof.

There exist different techniques for showing the determination of coefficients of
elliptic equations. One common way is to
view the DN map (for some elliptic PDE) as a pseudodifferential
operator, and to recover the material parameters at the boundary from its
symbol. This was first proposed by Sylvester and Uhlmann \cite{SU} for
the equation describing electrostatics,
\[\nabla\cdot(\gamma\nabla u)=0\,\,\text{with conductivity}\,\, 0<\gamma\in C^\infty(\overline\Omega).\]
The Dirichlet-to-Neumann map is defined by
\[\Lambda_\gamma:H^{1/2}(\partial\Omega)\ni\varphi\rightarrow \gamma\frac{\partial u}{\partial\nu}\in H^{-1/2}(\partial\Omega),\]
where $u$ is the solution of the above equation with $u=\varphi$ on $\partial\Omega$. The symbol of $\Lambda_\gamma$ simply has the leading order term (the principal symbol) 
\[\sigma(\Lambda_\gamma)(x',\xi')=\gamma(x')|\xi'|.\]
Here $(x',\xi')\in T^*(\partial\Omega)$ and $|\xi'|$ is the length of the cotangent vector with respect to the metric on $\partial\Omega$ from the Euclidean metric of $\mathbb{R}^3$.
It is almost immediate to recover $\gamma$ from
$\sigma(\Lambda_\gamma)$ \cite{SU}. The derivatives of $\gamma$ can be recovered
from the lower order terms of the full symbol. For elastostatics, the
reconstruction was given in \cite{NU3, NU} for the isotropic case and
in \cite{NTU} for the transversely isotropic case. The same approach
was also applied to (time-harmonic) Maxwell's equations
\cite{M, S}. We remark here that the calculation of the principal
symbol of the DN map for the elastic system is quite challenging.

In our previous paper \cite{dHNZ}, we show that via a finite-time Laplace transform, we can reduce the dynamical problem to an elliptic one: determine the isotropic $C^{ijkl}$ and $\rho$ at the boundary from $\Lambda^h$, where $\Lambda^h$ is the DN map corresponding to the elliptic system of equations
\begin{equation}\label{transformed eq}
\begin{cases}
\mathcal{M}v=\rho v - h^2\operatorname{div} (\mathbf{C}\varepsilon(v))=0~~\text{in}~\Omega,\\
v=\varphi~~\text{on}~\partial\Omega
\end{cases}
\end{equation}
with a parameter $h$ which is the reciprocal of the Laplace variable $\tau>0$, and
$\Lambda^h$ is defined by
\[\Lambda^h:  H^{1/2}(\partial\Omega)\ni\varphi\mapsto h\partial_L v=h(\mathbf{C}\varepsilon(v))\nu|_{\partial\Omega}\in H^{-1/2}(\partial\Omega).\]

In general, we do not have the exact $\Lambda^h$ from $\Lambda_T$. However, if we view $\Lambda^h$ as a semiclassical pseudodifferential operator with a small parameter $h$, we can recover the full symbol of $\Lambda^h$ from $\Lambda^T$. Then we expect to reconstruct the material parameters at the boundary from the full symbol of $\Lambda^h$. We refer to the book by Zworski \cite{Zworski} for an introduction to semiclassical pseudodifferential operators.

Generally, we believe that it is impossible to reconstruct a fully
anisotropic elastic tensor from the dynamical DN map. However, there
are some physically important symmetry restrictions -- while allowing
the presence of interfaces -- that are more general than isotropy, on
the stiffness tensor, under which we can still have an explicit
reconstruction. For an introduction of these symmetries, we refer to
Tanuma \cite{Tanuman} and Musgrave \cite{Mus}. In this paper, we will first survey to what extent we
can recover anisotropy.
 
%

We will give an explicit reconstruction formula of $C^{ijkl}$ at part of the boundary $\Sigma\subset\partial\Omega$, if either of the following three conditions holds:

\begin{enumerate}
\item $\Sigma$ is flat, $\mathbf{C}$ is \textit{transversely isotropic} (TI) with symmetry axis normal to $\Sigma$ (\textit{vertically transversely isotropic} VTI);
\item $\Sigma$ is flat, $\mathbf{C}$ is \textit{orthorhombic} with one of the three (known) symmetry planes tangential to $\Sigma$;
\item $\Sigma$ is curved, $\mathbf{C}$ and $\rho$ are constant.
\end{enumerate}

In elastostatics, Nakamura, Tanuma and Uhlmann \cite{NTU} gave an
explicit reconstruction scheme for the transversely isotropic
stiffness tensor assuming that the symmetry axis is tilted, that is,
not normal to the boundary (TTI), while the information is not enough
to recover the VTI case \cite{NT}. However, we can recover VTI elastic
parameters from the semiclassical symbol of $\Lambda^h$. Generally
speaking, this is because we have more information in dynamical data
than in static data. We will give further explanation in Section
\ref{rec}.

For the \textit{interior determination} from $\Lambda_T$ with $T$ large enough, uniqueness of smooth isotropic elastic tensor and density was shown under different geometrical conditions \cite{Rachelle1, Rachelle2, SUV4, Bhatta}.
We will study the \textit{interior determination} of piecewise analytic parameters based on our \textit{boundary determination} results,. For elliptic equations, the \textit{boundary determination} usually
leads to the uniqueness of \textit{interior determination} of
piecewise analytic coefficients. Kohn and Vogelius \cite{KV} first
established the relation in electrostatics. A recent paper by
C\^arstea, Honda and Nakamura \cite{CN} gives a uniqueness theorem for
piecewise constant stiffness tensors. The key in the proof is the continuation of the local elliptic DN map (see Section \ref{semi} for the definition). If the coefficients of elliptic equations are discontinuous, a variational argument is convenient for
this continuation. Ikehata \cite{Ike} gave such an argument in order to construct the physical parameters in an inclusion. In \cite{CN}, the authors adapted this variational argument for the continuation of the local elliptic DN map.  Runge's approximation plays an important role
in the continuation of data, which is in turn guaranteed by the Holmgren's uniqueness theorem. 

For
our problem, we need to know the exact operator $\Lambda^h$, not only
its full symbol. To get $\Lambda^h$, basically we need to have
$\Lambda_{T'}$ for any $T'$. This is possible by time continuation of
$\Lambda_T$, if $T$ is large enough, and the assumption that $\mathbf{C},\rho$ are
piecewise analytic. Also, with the exact $\Lambda^h$, we can view it
as a classical pseudodifferential operator. Under this classical
setting, we can also recover tilted transversely isotropic (TTI)
elastic parameters.

The time continuation is established with the boundary control (BC)
method. We basically follow the steps sketched by Kurylev and Lassas
\cite{KL}. The BC method was first introduced by Belishev
\cite{Beli}. Essentially, we need $T>2r$, where $r$ is the
approximate controllability time, and will be given in Lemma \ref{Holmgren}. With the assumption of piecewise analyticity, the existence of the approximate
controllability time is guaranteed by the unique continuation
principle (UCP) for lateral Cauchy data, which is essentially the
Holmgren-John uniqueness theorem. Indeed, relaxing the analyticity of the material parameters
  would require a very different method of proof. For acoustic wave equations, a uniqueness result
  for the piecewise smooth case under restrictive geometric conditions
  has been shown to be feasible \cite{CadaydHKU}.

The paper is organized in the following manner. In Section \ref{semi},
we show how to reduce the hyperbolic problem to an elliptic one, and
establish the relation between the dynamic $\Lambda_T$ to the symbol
of $\Lambda^h$. In Section \ref{cont}, we study the time continuation
of $\Lambda_T$ with piecewise analytic coefficients. In Section
\ref{symbol}, we introduce the boundary normal coordinates, and obtion
the symbol of $\Lambda^h$ in these coordinates via a factorization of
the operator $\mathcal{M}$ in (\ref{transformed eq}). 
%
Finally, in Section \ref{piecewise}, we show the uniqueness of interior determination for the
piecewise analytic material parameters.

\section{Transformation to an elliptic problem}\label{semi}${}$\newline\indent
In this section, we show how to reduce the hyperbolic problem
$(\ref{EQ no1})$ to the elliptic problem $(\ref{transformed eq})$. We
will give a modified exposition of what is given in
\cite{dHNZ}. Throughout this section, we assume that
$\mathbf{C},\rho\in L^\infty(\Omega)$. We consider the local DN
map. We introduce an open connected smooth part $\Sigma \subset \partial\Omega$.

For $r\ge0$ we let $H_{co}^{r}(\Sigma)$ be the closure in $H^{r}(\Sigma)$ of the set
\[C_c^\infty(\Sigma)=\{f\in C^\infty(\partial\Omega):~\text{supp}~f\subset\Sigma\},\]
and $H^{-r}(\Sigma)$ be its dual. We note that when
$\Sigma=\partial\Omega$, $H_{co}^{r}(\Sigma)=H^{r}(\Sigma)$. Then we define
the local DN map $\Lambda^{\Sigma}_{T}$ by
\[\Lambda^{\Sigma}_{T}:C^2([0,T];H_{co}^{1/2}(\Sigma))\ni f\mapsto \mathbf{C}\varepsilon(u)\nu\vert_{\Sigma\times[0,T]}\in L^2([0,T];H^{-1/2}(\Sigma)),\]
where $u$ solves (\ref{EQ no1}).

We also define the local DN map $\Lambda^{h,\Sigma}$ for the elliptic problem $(\ref{transformed eq})$ by
\[\Lambda^{h,\Sigma}:H_{co}^{1/2}(\Sigma)\ni \varphi\mapsto h\mathbf{C}\varepsilon(v)\nu\vert_{\Sigma}\in H^{-1/2}(\Sigma),\]
where $v$ solves the equation (\ref{transformed eq}).
We let $\psi\in H^{1/2}_{co}(\Sigma)$, $\chi(t)=t^2$, and $f(x,t)=\chi(t)\psi(x)$. We take $u_0\in H^{1}(\Omega)$ (by inverse trace theorem) such that $u_0=\psi$ on $\partial\Omega$ and satisfies
\[
\operatorname{div} (\mathbf{C}\varepsilon(u_0))=0~~\text{in}~\Omega
\]
 with the estimate
\[\|u_0\|_{H^1(\Omega)}\leq C\|\psi\|_{H^{1/2}_{co}(\Sigma)}. \]
Then we seek a solution $u$ of \eqref{EQ no1} in the form
\[u(y,t)=\chi(t) u_0(y)+ u_1(y,t),\]
where $u_1(\cdot,t)\in L^2((0,T);H^1_0(\Omega))$ with 
\[\partial_t u_1(\cdot,t)\in L^2((0,T); L^2(\Omega)), \partial_t^2 u_1\in L^2((0,T); H^{-1}(\Omega))\]
solves the following system in the weak sense,
\begin{equation}\label{eq no2}
\begin{cases}
\rho\partial^2_tu_1-\operatorname{div} (\mathbf{C}\varepsilon(u_1))= F(y,t)~~\text{in}~\Omega_T=\Omega\times(0,T) ,\\
u_1=0~~\text{on}~\Sigma=\partial\Omega\times(0,T) ,\\
u_1(y,0)=\partial_t u_1(y,0)=0~~\text{in}~\Omega,
\end{cases}
\end{equation}
where
\[F(y,t)=-2\rho u_0\in L^2((0,T);L^2(\Omega)).\]
Problem \eqref{eq no2} is equivalent to solving for $u_1\in L^2((0,T);H^1_0(\Omega))$ with $\partial_t u_1(\cdot,t)\in L^2((0,T); L^2(\Omega))$, $\partial_t^2 u_1\in L^2((0,T); H^{-1}(\Omega))$ which satisfies 
\begin{equation}\label{weak form}
-\int_0^T (\rho\partial_t u_1, \partial_t v)\,dt+\int_0^T B[u_1(\cdot,t),v(\cdot,t)]\,dt=\int_0^T (F(\cdot,t), v(\cdot,t))\,dt
\end{equation}
for any $v\in C_0^\infty([0,T]; H_0^1(\Omega))$, where $(\cdot\,,\cdot)$ is the $L^2(\Omega)$ inner product, 
\[B[\varphi,\psi]=\int_\Omega \mathbf{C}\varepsilon(\varphi)::\varepsilon(\psi)\mathrm{d}y,\,\,\varphi, \psi\in H^1_0(\Omega),\]
and the notation $::$ denotes the inner product of matrices.

It is well known (cf. \cite{LM}) that there exists a unique solution
$u_1\in L^2((0,T);H^1_0(\Omega))$ of \eqref{weak form} with $\partial_tu_1\in L^2((0,T);L^2(\Omega))$, $\partial_t^2 u_1\in L^2((0,T); H^{-1}(\Omega))$.
By possibly modifying the value of $u_1$ in a zero-measure set,
\[u_1\in C^1([0,T],H^{-1}(\Omega))\cap C^0([0,T],L^2(\Omega))\] while it satisfies the estimate
\begin{equation}
\|u_1(\cdot,t)\|_{L^2(\Omega)}+\|\partial_tu_1(\cdot, t)\|_{H^{-1}(\Omega)}\leq C\|F\|_{L^2((0,T);L^2(\Omega))},\,\,t\in[0,T]
\end{equation}
(see \cite{Evans, LM} for the details of these).
Therefore,
\[
\int_0^T \partial_t^2u_1(\cdot, t)e^{-\tau t}\mathrm{d}t=\tau^2\int_0^Tu_1(\cdot, t)e^{-\tau t}\mathrm{d}t+e^{-\tau T}(\partial_t u_1(\cdot,T)+\tau u_1(\cdot, T)) ~~\text{in}~H^{-1}(\Omega)
\]
with the estimate
\[\begin{split}
\|u_1(\cdot, T)\|_{L^2(\Omega)}+\|\partial_tu_1(\cdot,T)\|_{H^{-1}(\Omega)}\leq &C\|F\|_{L^2([0,T];L^2(\Omega))}\\
\leq & CT^{1/2}\|\psi\|_{H_{co}^{1/2}(\Sigma)}.
\end{split}\]

Based on this observation consider the finite-time Laplace transform $w(\cdot,\tau)$ of $u_1$:
\[w(\cdot,\tau)=\int_0^T u_1(\cdot, t)e^{-\tau t}\mathrm{d}t.\]
Then
\[( \rho \tau^2w(\cdot,\tau), v)+B[w(\cdot,\tau),v]=( F_1,v),\,\,v\in H_0^1(\Omega)\]
with
\[F_1:=\int_0^T Fe^{-\tau t}\mathrm{d}t-e^{-\tau T}(\partial_t u_1(\cdot,T)+\tau u_1(\cdot,T))\in H^{-1}(\Omega),\]
that is, $w$ satisfies the elliptic equation
\[\begin{cases}
\rho \tau^2w-\text{div}(\mathbf{C}\varepsilon(w))=F_1,\\
w=0~~\text{on}~\partial\Omega
\end{cases}\]
 in the weak sense.

Now let $v$ satisfy (\ref{transformed eq}) with $h=1/\tau$ and the Dirichlet data $\varphi$ taken as
\[\varphi=\psi \chi_1(\tau;T)\,\,\text{with}\,\, \chi_1(\tau;T)=\int_0^T t^2e^{-\tau t}\mathrm{d}t\]
Then we will estimate 
\[r(y,\tau)=v(y,\tau)-\int_0^Tu(y,t)e^{-\tau t}\mathrm{d}t.\]
By a direct computation, $\chi_1$ satisfies the estimate
\[\chi_1(\tau;T)\geq \frac{C}{\tau^3}\]
for some $C>0$ independent of $\tau$ and $T$.
Furthermore, $z=v-u_0\chi_1(\tau;T)$ satisfies
\[\begin{cases}
\rho \tau^2z-\text{div}(\mathbf{C}\varepsilon(z))=-\chi_1(\tau;T)\rho\tau^2 u_0,\\
z=0~~\text{on}~\partial\Omega.
\end{cases}\]
We observe that $r=z-w$ and that it
satisfies
\[\begin{cases}
\rho \tau^2r-\text{div}(\mathbf{C}\varepsilon(r))=e^{-\tau T}(\partial_t u_1(T)+\tau u_1(T))+(2Te^{-\tau T}+\tau T^2e^{-\tau T})u_0,\\
r=0~~\text{on}~\partial\Omega.
\end{cases}\]
Then we have
\[\|r(\cdot,\tau)\|_{H^1(\Omega)}\leq C\tau T^3e^{- \tau T}\|\psi\|_{H^{1/2}_{co}(\Sigma)}\]
with $C$ independent of $\tau$ and $T$ by the standard elliptic regularity estimate.

Now consider the finite-time Laplace transform $\mathcal{L}_T u$ of $u$ given as
\[\mathcal{L}_Tu=\int_0^Tue^{-\tau t}\mathrm{d}t,~\text{with}~\tau>0.\]
Then we have
\[\|\partial_L v-\mathcal{L}_T\partial_L u\|_{H^{-1/2}(\Sigma)}\leq C\tau T^3e^{- \tau T}\|\psi\|_{H_{co}^{1/2}(\Sigma)}\]
or, equivalently,
\[\|\Lambda^{h,\Sigma}\varphi-h\mathcal{L}_T\Lambda_T^\Sigma\chi\chi_1^{-1}\varphi\|_{H^{-1/2}(\Sigma)} \leq C\left(\frac{T}{h}\right)^3e^{-\frac{T}{h}}\|\varphi\|_{H_{co}^{1/2}(\Sigma)}.\]
Hence,
\begin{equation}\label{op_estimates}
\|\Lambda^{h,\Sigma}-h\mathcal{L}_T\Lambda_T^\Sigma\chi\chi_1^{-1}\|_{H^{1/2}_{co}(\Sigma)\rightarrow H^{-1/2}(\Sigma)}\leq   C\left(\frac{T}{h}\right)^3e^{-\frac{T}{h}},
\end{equation}
which means that for a fixed $T>0$, 
\[\Lambda^{h,\Sigma}\sim h\mathcal{L}_T\Lambda_T^\Sigma\chi\chi_1^{-1}\]
modulo an operator mapping $H^{1/2}_{co}(\Sigma)$ to $H^{-1/2}(\Sigma)$ with estimates $\mathcal{O}(h^\infty)$. Thus, from
$\mathcal{L}_T$, we can obtain the full symbol of $\Lambda^{h,\Sigma}$ viewed as a semiclassical pseudodifferential operator with a small parameter $h$.

We remark here that, in general, we could not have the full operator
$\Lambda^{h,\Sigma}$ from $\Lambda_T^\Sigma$ for a finite $T$, but we can get the full
symbol of $\Lambda^{h,\Sigma}$ from which we can already expect to recover the
material parameters at the boundary. Later, we will see in Section \ref{cont} that, we can get $\Lambda^{h,\Sigma}$ from
$\Lambda_{T^*}^\Sigma$ for some $T^*$ large enough, and the material parameters are
piecewise analytic. This
enables us to recover piecewise analytic densities and stiffness
tensors.

\section{Time continuation of the DN map}\label{cont}${}$\newline\indent
In this section, we show that we can obtain $\Lambda_T^\Sigma$ for any $T>0$
from $\Lambda_{T^*}^\Sigma$ for a fixed $T^*$ large enough, assuming that the
coefficients are piecewise analytic. We will follow \cite{KL}.

We assume that $\Omega$ consists of a finite number of Lipschitz
subdomains $D_\alpha$, $\alpha=1,\cdots,K$. That is,
$\overline{\Omega}=\cup_{i=1}^K\overline{D}_\alpha$, $D_\alpha\cap
D_\beta=\emptyset$ if $\alpha\neq \beta$.  We also assume that in each
$D_\alpha$, $\mathbf{C}$ and $\rho$ are analytic up to its boundary.  Since $\Omega$ is a
domain, we can assume without loss of generality that there exist
smooth nonempty
$\Sigma_{\alpha+1}\subset\overline{D_\alpha}\cap\overline{D_{\alpha+1}}$,
$\alpha=1,\cdots,K$ with $\Sigma=\Sigma_1\subset\partial\Omega$. First, we prove the following global version of the Holmgren-John uniqueness theorem.
\begin{lemma}\label{Holmgren}
There exists a finite $r>0$, such that, for any $t\geq 2r$, if $e\in\mathcal{D}'((0,t)\times\Omega)$ satisfies
\[
\begin{split}
\rho\partial^2_te=\operatorname{div} (\mathbf{C}\varepsilon(e))~~\text{in}~\Omega\times(0,t),\\
e\vert_{\Sigma\times[0,t]}=(\mathbf{C}\varepsilon(e))\nu \vert_{\Sigma\times[0,t]}=0,
\end{split}
\]
then $e(\frac{t}{2})=\partial_te(\frac{t}{2})=0$ in $\Omega$. 
\end{lemma}

\noindent
Here and in the remainder of this section we will suppress the space
coordinates in our notation.

\begin{proof}
First, we have by the standard Holmgren's theorem (cf. \cite{Treves})
that $e$ vanishes on $D_1\times [\varrho_1,t-\varrho_1]$ for some
$\varrho_1>0$, if $\frac{t}{2}>\varrho_1$. For the unique continuation
across the interfaces, we follow the reasoning in \cite{Oksanen} in the following argument. First,
we apply an analytic continuation of $\mathbf{C},\rho$ on $D_2$ to a
small neighborhood $U_2$ of $\Sigma_2$. Use
$\mathbf{C}_{D_2},\rho_{D_2}$ to denote the extended coefficients on
$\tilde{D}_2=D_2\cup U_2$. Now, $e$ satisfies
\[
\rho_{D_2}\partial^2_te=\operatorname{div} (\mathbf{C}_{D_2}\varepsilon(e))~~\text{in}~\tilde{D}_2\times(\varrho_1,t-\varrho_1).
\]
Since $e$ vanishes on $\tilde{D}_2\cap D_1$, we can apply Holmgren's theorem again to conclude that $e$ vanishes on $D_2\times (\varrho_1+\varrho_2,t-\varrho_1-\varrho_2)$, for some $\varrho_2>0$. We need $\frac{t}{2}>\varrho_1+\varrho_2$. We can repeat the process and prove the lemma provided that $r$ is sufficiently large.
\end{proof}
~\\

Let $u^f$ be the solution of (\ref{EQ no1}) with boundary value $f$.\\

\begin{lemma}\label{dense}
The pairs $(u^f(2r),-\partial_tu^f(2r))$, $f\in C_c^\infty(\Sigma\times(0,2r))$ are dense in $ H_0^1(\Omega)\times L^2(\Omega)$.
\end{lemma}
\begin{proof}
Assume that a pair
\[
(\alpha,\beta)\in (H_0^1(\Omega)\times L^2(\Omega))'=H^{-1}(\Omega)\times L^2(\Omega)
\]
satisfies
\[\langle \alpha, u^f(2r)\rangle_{( H^{-1}(\Omega),H_0^1(\Omega) )}+\langle \beta,-\partial_tu^f(2r)\rangle_{L^2(\Omega)}=0\]
for all $f\in C_c^\infty(\Sigma\times(0,2r))$, where $\langle\,\cdot\,,\,\cdot\, \rangle_{(H^{-1}(\Omega), H_0^{1}(\Omega))}$ is a pairing between an element in $H^{-1}(\Omega)$ and an element in
$H^1(\Omega)$ defined as a continuous extension of the $L^2(\Omega)$ inner product. It is sufficient to show that
\[\alpha=\beta=0.\]
Let $e$ be the unique solution of
\begin{equation}\label{adj}
\begin{cases}
\rho\partial^2_te=\operatorname{div} (\mathbf{C}\varepsilon(e))~~\text{in}~\Omega\times(0,2r),\\
e=0~~\text{on}~\partial\Omega\times(0,2r) ,\\
 \rho e(y,2r)=\beta,~\rho\partial_te(y,2r)=\alpha~~\text{in}~\Omega
\end{cases}
\end{equation}
with 
\[e\in C([0,2r];L^2(\Omega)),~\partial_te\in C([0,2r];H^{-1}(\Omega)).\]
We note that the well-posedness of the above problem was established in \cite{LM}.

Upon integration by parts, we obtain
\[
\begin{split}
0=&\int_0^{2r}\big(\langle (\rho\partial_t^2 e-\operatorname{div} (\mathbf{C}\varepsilon(e)),u^f\rangle_{(H^{-1}(\Omega),H_0^1(\Omega))}\\
&\quad\quad\quad\quad\quad\quad\quad\quad-\langle(\rho\partial_t^2 u^f-\operatorname{div} (\mathbf{C}\varepsilon(u^f))),e\rangle_{(H^{-1}(\Omega),H_0^1(\Omega))}\big)\mathrm{d}t\\
=&\langle\beta, \partial_tu^f(2r)\rangle_{L^2(\Omega)}- \langle \alpha,u^f(2r)\rangle_{(H^{-1}(\Omega),H_0^1(\Omega))}\\
&\quad\quad\quad\quad\quad\quad\quad\quad-\int_0^{2r}\langle(\mathbf{C}\varepsilon(e))\nu,f\rangle_{(H^{-1/2}(\Sigma),H^{1/2}(\Sigma))}\mathrm{d}t\\
=&-\int_0^{2r}\langle(\mathbf{C}\varepsilon(e))\nu,f\rangle_{(H^{-1/2}(\Sigma),H^{1/2}(\Sigma))}\mathrm{d}t
\end{split}
\]
for any $f\in C_c^\infty(\Sigma\times(0,2r))$, where $\langle\,\cdot\,,\,\cdot\, \rangle_{(H^{-1/2}(\Omega), H^{1/2}(\Omega))}$ is defined likewise\\ $\langle\,\cdot\,,\,\cdot\, \rangle_{(H^{-1}(\Omega), H_0^{1}(\Omega))}$. Hence we have
\[\int_0^{2r}\langle(\mathbf{C}\varepsilon(e))\nu,f\rangle_{(H^{-1/2}(\Sigma),H^{1/2}(\Sigma))}\mathrm{d}t=0.\]
This yields
\[e\vert_{\Sigma\times[0,2r]}=(\mathbf{C}\varepsilon(e))\nu \vert_{\Sigma\times[0,2r]}=0.\]
By Lemma \ref{Holmgren}, we have
\[e(r)=\partial_te(r)=0,~~\text{on}~\Omega.\]
Thus $e=0$ on $\Omega\times [0,2r]$ and, hence, $\alpha=\beta=0$.
\end{proof}
~\\

We consider a bilinear form
\[E(u^f,u^g,t)=\int_\Omega \rho\partial_tu^f(t)\partial_tu^g(t)+\mathbf{C}\varepsilon(u^f(t))::\varepsilon(u^g(t))\mathrm{d}y.\]
To simplify the notation, we write $E(u^f,t)=E(u^f,u^f,t)$.\\

\begin{lemma}\label{deter}
The operator $\Lambda_t^\Sigma$ determines $E(u^f,u^g,t)$ for $f,g\in C_c^\infty(\Sigma\times(0,t))$.
\end{lemma}
\begin{proof}
By the estimates for $\partial_tu^f$,$\partial_t^2u^f$, we have 
\[u^f\in C^1([0,t];H^1_0(\Omega))\cap C^2([0,t];L^2(\Omega)). \]
Integrating by parts, we find that
\[
\begin{split}
\partial_t E(u^f,t)
=& 2\int_\Omega \rho\partial_t^2u^f(t)\partial_tu^f(t)+\mathbf{C}\varepsilon(u^f(t))::\varepsilon(\partial_tu^f(t))\mathrm{d}y\\
=&2\langle\mathbf{C}\varepsilon(u^f(t))\nu,\partial_tu^f(t)\rangle_{(H^{-1/2}(\Sigma),H^{1/2}_{co}(\Sigma))}\\
=&2\langle (\Lambda_t^\Sigma f)(t),\partial_tf(t)\rangle_{(H^{-1/2}(\Sigma),H^{1/2}_{co}(\Sigma))}.
\end{split}
\]
With the initial conditions, $E(u^f,0)=0$, and we can determine $E(u^f,t)$ as well as
\[
E(u^f,u^g,t)=\frac{1}{4}(E(u^{f+g},t)-E(u^{f-g},t))
\]
by polarization.
\end{proof}
~\\

We arrive at\\

\begin{theorem}
Let $T^*>2r$, then $\Lambda_{T^*}^\Sigma$ determines $\Lambda_T^\Sigma$ for any $T>0$.
\end{theorem}

\begin{proof}
Let
$\delta =\frac{T^*-2r}{2}$. It is sufficient to show $\Lambda_{T^*}^\Sigma$
determines $\Lambda_{T^*+\delta}^\Sigma$. Indeed, by repeating the process
presented below, we obtain the result.

For any $f\in C^\infty([0,T^*+\delta],H^{1/2}_{co}(\Sigma))$, take a decomposition $f=g+h$, where $g\in C_c^\infty([0,2\delta),H^{1/2}_{co}(\Sigma))$ and $h\in C_c^\infty((\delta,T^*+\delta],H^{1/2}_{co}(\Sigma))$. Since we have $(\Lambda_{T^*+\delta}^\Sigma h)(t)=(\Lambda_{T^*}^\Sigma Y_{-\delta}h)(t)$ with $Y_{-\delta}h(t):=h(t+\delta)$ for $t\in [0,T^*]$ and
\[
\begin{split}
\Lambda^\Sigma_{T^*+\delta} f=\Lambda^\Sigma_{T^*+\delta}g+\Lambda^\Sigma_{T^*+\delta}h,
\end{split}
\]
we only need to show that $\Lambda_{T^*}^\Sigma$ determines $(\Lambda^\Sigma_{T^*+\delta}g)(t)$ for any $t\in(T^*,T^*+\delta]$.

Let $t_0=2r+\delta$. By Lemma \ref{dense}, there are $g_n\in C_c^\infty(\Sigma\times (0,2r))$ such that
\begin{equation}\label{et0}
\lim_{n\rightarrow\infty}(u^{g_n}(2r),\partial_tu^{g_n}(2r))=(u^{g}(t_0),\partial_tu^g(t_0))
\end{equation}
in the $H^1_0(\Omega)\times L^2(\Omega)$ topology. It is straightforward to show that $(\ref{et0})$ is equivalent to
\begin{eqnarray}
\lim_{n\rightarrow\infty}E(u^{\tilde{g}_n},t_0)=0.\label{et1}
\end{eqnarray}
Here, $\tilde{g}_n(t)=g(t)-g_n(t-\delta)$ with $g_n(s)=0,-\delta<s<0$. 
%
By Lemma \ref{deter}, we can construct $g_n$ using only $\Lambda_{T^*}^\Sigma$ to construct $g_n$ satisfying $(\ref{et0})$. The functions $y_n(t):=u^{g_n}(t)$ for $t\in[2r,T^*]$ are the solutions of the initial boundary value problem,
\[
\begin{split}
&\rho\partial_t^2y_n=\operatorname{div} (\mathbf{C}\varepsilon(y_n)~~\text{in}~\Omega\times[2r,T^*],\\
&y_n\vert_{\partial\Omega\times[2r,T^*]}=0,~y_n(2r)=u^{g_n}(2r),~~\partial_ty_n(2r)=\partial_tu^{g_n}(2r).
\end{split}
\]
We note that $y(t):=u^g(t+\delta)$ satisfies the same equation with initial data
\[y(2r)=u^{g}(t_0),~~\partial_ty(2r)=\partial_tu^{g}(t_0).\]
Also by the continuous dependence of solutions on initial data, we have
\[\lim_{n\rightarrow\infty}\mathbf{C}\varepsilon(y_n)\nu|_{\Sigma\times[2r,T^*]}=\mathbf{C}\varepsilon(y)\nu|_{\Sigma\times[2r,T^*]}\]
in the $L^2$ topology. Hence,
\begin{equation}
\begin{array}{lll}
(\Lambda_{T^*+\delta}^\Sigma g)(t)=C\epsilon(u^g(t))\nu\vert_{\Sigma\times[t_0,T^*+\delta]}=\mathbf{C}\varepsilon(y(t-\delta))\nu\vert_{\Sigma\times[2r,T^*]}\\
\qquad\qquad=\lim_{n\rightarrow\infty}\mathbf{C}\varepsilon(y_n(t-\delta))\nu\vert_{\Sigma\times[2r,T^*]}=\lim_{n\rightarrow\infty}(\Lambda_{T^*}^\Sigma Y_\delta\, y_n)(t)\\
\qquad\qquad=\big(Y_{-\delta}(\lim_{n\rightarrow\infty}\Lambda_{T^*}^\Sigma g_n)\big)(t)~~\text{for}~t\in[t_0,T^*+\delta]
\end{array}
\end{equation}
and we can determine $(\Lambda^\Sigma_{T^*+\delta}g)(t)$ on
$[T^*,T^*+\delta]$ from $\Lambda^\Sigma_{T^*}g_n$.
\end{proof}

\section{The principal symbol of $\Lambda^{h,\Sigma}$}\label{symbol}${}$\newline\indent
We now analyze the principal symbol of $\Lambda^{h,\Sigma}$ as a semiclassical pseudodifferential operator. All the calculations of semiclassical pseudodifferential symbols can be found
in \cite{dHNZ}. We will sketch the key points in the following.

For the analysis, we need to introduce the boundary normal coordinates.
Given a boundary
point $p_0\in\Sigma$, let $(x^1(p'),x^2(p'))$ be local
coordinates of $\Sigma$ close to $p_0$. For any $p$ near
$p_0$, we use the boundary normal coordinates
$x(p)=(x^1(p'),x^2(p'),x^3)$, where $p'$ is the nearest point on
$\Sigma$ to $p$ with $x(p')=(x^1(p),x^2(p),0)$ and
$x^3=\text{dist}(p,p')$. Here the distance function $\text{dist}(\cdot,\cdot)$ is respect to the Euclidean metric. Thus, $\Sigma$ is locally represented
by $x^3=0$. We let $(\xi_1,\xi_2,\xi_3)$ and $(\eta_1,\eta_2,\eta_3)$
represent the same conormal vector with respect to different
coordinates, $(x^1,x^2,x^3)$ and $(y^1,y^2,y^3)$, such that
$\xi_\alpha \mathrm{d}x^\alpha=\eta_i \mathrm{d}y^i$ using the Einstein summation convention, which will be repeatedly used in the paper. Here $y$ denotes the Cartesian coordinates introduced before.
We introduce the coordinate mapping, $F$, as
\[F(y^1(p),y^2(p),y^3(p))=(x^1(p),x^2(p),x^3(p))\]
and  the Jacobian
\begin{equation}\label{Jacobian}
J^a_{~i}=\left(\frac{\partial x^a}{\partial y^i}\right).
\end{equation}
Then $J^a_{~i}\xi_a=\eta_i$ (or equivalently, $J^T\xi=\eta$), and
\[ \tilde{C}^{abcd}(p)=J^a_{~i}J^b_{~j}J^c_{~k}J^d_{~l}C^{ijkl}(p)
\quad\text{and}\quad G^{ab}=J^a_{~i}J^b_{~j}\delta^{ij}=:G^{-1}.\]
Here, $G=(G_{ab})$ is the induced Riemannian metric for boundary normal
coordinates, $x$. Also,
\[J^a_{~i}\tilde{v}_a=v_i.\]
In the boundary normal coordinates, (\ref{transformed eq}) attains the
form
\begin{equation}\label{eq.1 tensorial}
\begin{cases}
(\tilde{\mathcal{M}}\tilde{v})^a=\rho G^{ac}\tilde{v}_c-h^2\displaystyle\sum_{b,c,d=1}^3\nabla_b (\tilde{C}^{abcd}\varepsilon_{cd}(\tilde{v}))=0~\text{in}~\{x^3>0\}~\text{for}~1\leq a\leq 3,\\
\tilde{v}_d|_{x^3=0}=\tilde{\psi}_d,~~1\leq d\leq 3,
\end{cases}
\end{equation}
where $\nabla_a$ is the covariant derivative with respect to metric $G$ and  $\varepsilon_{cd}(\tilde{v})=\frac{1}{2}(\nabla_c \tilde{v}_d+\nabla_d \tilde{v}_c)$.

We express $\Lambda^{h,\Sigma}$ in boundary normal coordinates as
\[\Lambda^{h,\Sigma}:\tilde{\psi}_d\rightarrow h\tilde{C}^{a3cd}\varepsilon_{cd}(\tilde{v})\vert_{\Sigma}.\]
Here we denote $\xi=(\xi',\xi_3)=(\xi_1,\xi_2,\xi_3)$. Then $\Lambda^{h,\Sigma}$ is a semiclassical pseudodifferential operator with full symbol $\tilde{\sigma}(\Lambda^{h,\Sigma})(x',\xi')$ which has the asymptotics \cite{dHNZ}
\[\tilde{\sigma}(\Lambda^{h,\Sigma})(x',\xi')=\sum_{j\geq 0}h^{j}\lambda_{-j}(x',\xi').\]
In this expansion, $\lambda_0(x',\xi')$ signifies the principal symbol of $\Lambda^{h,\Sigma}$. We proceed with calculating $\lambda_0(x',\xi')$.

We define
\begin{equation}\label{QRD}
\begin{split}
\tilde{Q}(x,\xi') &= \left(\sum_{b,d=1}^2\tilde{C}^{abcd}(x)\xi_b\xi_d;~1\le a,\, c\,\le 3 \right) ,\\
\tilde{R}(x,\xi') &= \left(\sum_{b=1}^2\tilde{C}^{abc3}(x)\xi_b;~1\le a,\, c\,\le 3\right) ,\\[0.25cm]
\tilde{D}(x) &=\left(\tilde{C}^{a3c3}(x);~1\le a,\, c\,\le 3 \right)
\end{split}
\end{equation}
and then
\begin{equation}\label{mathcal M}
\tilde{M}(x,\xi)=\tilde{D}(x)\xi_3^2+(\tilde{R}(x,\xi')+\tilde{R}^T(x,\xi'))\xi_3+\tilde{Q}(x,\xi')+\rho(x)G
\end{equation}
is the principal symbol of $\tilde{\mathcal{M}}$. First, we introduce the following factorization of $\tilde{M}$.

We note that $\tilde{M}(x,\xi)$ is a positive definite matrix for
$x \in \overline{\Omega}$, $\xi \in \mathbb{R}^3 \backslash
0$.  Hence, for fixed $(x,\xi')$,
$\det \tilde{D}^{-1/2}\tilde{M}(x,\xi)\tilde{D}^{-1/2} = 0$
 in $\xi_3$
admits $3$ roots $\xi_3 = \zeta_j~(j=1,2,3)$ with positive imaginary
parts and $3$ roots $\overline{\zeta_j}~(j=1,2,3)$ with negative
imaginary parts. Thus,

\begin{lemma}[\cite{GLR}]\label{L1}
There is a unique factorization
\[\check{M}(x,\xi)=\tilde{D}(x)^{-1/2} \tilde{M}(x,\xi) \tilde{D}(x)^{-1/2}
=(\xi_3-\check{S}_0^*(x,\xi'))(\xi_3-\check{S}_0(x,\xi')),\]
with $\operatorname{Spec}(\check{S}_0(x,\xi')) \subset \mathbb{C}_+$, where $\operatorname{Spec}(\check{S}_0(x,\xi'))$ is the spectrum of $\check{S}_0(x,\xi')$.
In the above,
\[\check{S}_0(x,\xi'):=\left(\oint_\gamma\zeta \check{M}(x,\xi',\zeta)^{-1}\mathrm{d}\zeta\right)\left(\oint_\gamma \check{M}(x,\xi',\zeta)^{-1}\mathrm{d}\zeta\right)^{-1},\]
where $\gamma\subset\mathbb{C}_+:=\{\zeta\in\mathbb{C}:\mathrm{Im}\,{\zeta}:=\text{\rm imaginary part of $\zeta$}>0\}$ is a continuous curve enclosing all the roots $\zeta_j~(j=1,2,3)$
of $\text{\rm det}(\check{M}(x,\xi',\zeta))=0$ in $\zeta\in\mathbb{C}_+$.
\end{lemma}

This lemma implies that the following factorization of $\tilde{M}(x,\xi)$,
\begin{equation}\label{fac M}
\tilde{M}(x,\xi)=(\xi_3-\tilde{S}_0^*(x,\xi'))\tilde{D}(x)(\xi_3-\tilde{S}_0(x,\xi')),
\end{equation}
where
\[\tilde{S}_0(x,\xi') = \tilde{D}^{-1/2}(x)\check{S}_0(x,\xi')\tilde{D}^{1/2}(x),\]
\begin{equation}
\lambda_0(x',\xi') = -{\rm i}(\tilde{D}(x',0)\tilde{S}_0(x',0,\xi')+\tilde{R}(x',0,\xi'))
\end{equation}
by the definition of $\Lambda^{h,\Sigma}$. 

By \cite[Proposition~3.4]{dHNZ}, we obtain $D^\alpha_{x^3} \lambda_0$
for any $\alpha=1,2,\cdots$ from the lower order terms $\lambda_{-j}$,
$j=1,2,\cdots$ of the full symbol of $\Lambda^{h,\Sigma}$. In order to give an
explicit reconstruction of the material parameters at the boundary, we
need to calculate the closed form of $\lambda_0(x',\xi')$.

\subsection*{Surface impedance tensor}\label{surface}${}$\indent

We need to have the explicit closed form of the principal symbol $\lambda_0$. The calculation in boundary normal coordinates would be extremely unclear. In this part, we
establish the relation between the principle symbol $\lambda_0$, which is defined in boundary normal coordinates $x$, and
the so-called surface impedance tensor $Z$, which is defined in Cartesian coordinates $y$. A similar discussion can be found in Section 4 of \cite{dHNZ}. 

Take $\mathbf{n}$ to be
the outer normal direction at $p\in\Sigma$ expressed in Cartesian coordinate. Denote
$\mathbf{n}=(n_1,n_2,n_3)$. Let
$\mathbf{m}=(m_1,m_2,m_3)$ be a vector (not a unit one) normal to
$\mathbf{n}$. We have
\begin{equation}\label{relation}
J^{-T}\mathbf{n}(p)=(0,0,1),\,\, J^{-T}\mathbf{m}(p)=(\xi'(p),0)=(\xi_1(p),\xi_2(p),0)\,
\end{equation}
with the Jacobian $J=(J_i^a)$, defined in (\ref{Jacobian}), at $p$. 

The operator $\mathcal{M}$ in (\ref{transformed eq}) has principal symbol $M=(M^{ik}(p,\eta))$ at $p$ given by
\[M^{ik}(p,\eta)=\sum_{j,l=1}^3 C^{ijkl}(p)\eta_j\eta_l+\rho\delta^{ik}\]
in $y$ coordinates, and operator $\tilde{\mathcal{M}}$ in (\ref{eq.1 tensorial}) has principal symbol $\tilde{M}$
\[\tilde{M}^{ac}(p,\xi)=\sum_{b,d=1}^3 \tilde{C}^{abcd}(p)\xi_b\xi_d+\rho G^{ac},\]
in $x$ coordinates. Using the transformation rules of tensors, we have
\[J^a_{~i} M^{ik} (p,\eta) J^c_{~k}=J^a_{~i}J^b_{~j}J^c_{~k}J^d_{~l}C^{ijkl}(p)\xi_b\xi_d+J^a_{~i}J^c_{~k}\rho\delta^{ik},\]
which is nothing but 
\[JMJ^T=\tilde{M}.\]

We choose $\eta=q\mathbf{n}+\mathbf{m}=(qn_1+m_1,qn_2+m_2,qn_3+m_3)$ so that $\xi=J^{-T}(q\mathbf{n}+\mathbf{m})=(\xi_1,\xi_2,q)$. It follows that
\[J^{-1}\tilde{M}(p,\xi)J^{-T}=M(p,q\mathbf{n}+\mathbf{m}).\]
We obtain
\[M(p,q\mathbf{n}+\mathbf{m})=Dq^2+\left(R+R^T\right)q+Q+\rho\]
with
\begin{equation}
\begin{array}{rcl}
   D(\mathbf{n}) &=& \left(\displaystyle \sum_{j,l=1}^3
       C^{ijkl} n_j n_l; 1\le i,\, k\,\le 3\right) ,
\\
   R(\mathbf{n},\mathbf{m}) &=& \left(\displaystyle \sum_{j,l=1}^3
       C^{ijkl} m_j n_l; 1\le i,\, k\,\le 3\right) ,
\\
   Q (\mathbf{m}) &=& \left(\displaystyle \sum_{j,l=1}^3
       C^{ijkl} m_j m_l; 1\le i,\, k\,\le 3\right) .
\end{array}
\end{equation}

Similar to Lemma \ref{L1}, there is a unique factorization of $M$,
that is,
\[M(p,q\mathbf{n}+\mathbf{m})=\left(q-S_0^*\right)D(q-S_0), ~~~\text{Spec}\left(S_0(\mathbf{n},\mathbf{m})\right)\subset\mathbb{C}_+,\]
where $S_0(\mathbf{n},\mathbf{m})$ is independent of $q$. Changing coordinates,
\[\begin{split}\tilde{M}(p,\xi)&=JMJ^T=J(q-S_0^*)D(q-S_0)J^T\\
&=\left(q-JS_0^*J^{-1}\right)(JDJ^T)\left(q-J^{-T}S_0J^T\right).
\end{split}\]
Hence, by the fact $Spec(J^{-T}S_0 J^T)\subset {\mathbb C}_+$ and the uniqueness of the factorization, 
\[
   \tilde{S}_0 = J^{-T}S_0J^T.
\]

We define the surface impedance tensor $Z=Z(p,\mathbf{m},\mathbf{n})$ by
\[
   Z(p,\mathbf{m},\mathbf{n}) = -\mathrm{i}(DS_0+R^T).
\]
Based on the previous arguments, we can now express the principal symbol, $\lambda_0$, in terms of $Z$:

\begin{lemma}
The principal symbol $\lambda_0(x'(p),\xi')$ is related to the surface
impedance tensor as
\begin{equation}
\lambda_0(x'(p),\xi')=J Z(p,\mathbf{m},\mathbf{n}) J^T,
\end{equation}
where the relation between $\xi'$ and $\mathbf{n},\mathbf{m}$ is defined in $(\ref{relation})$.
\end{lemma}
~\\

The reconstruction of the density and stiffness tensor for the
principal symbol is now simplified to a reconstruction from the
surface impedance tensor. The same applies to their derivatives.


\section{Recovery of the material parameters}\label{piecewise}
\subsection{Recovery at the boundary}${}$\newline\indent
In this subsection, we summarize our results on recovering of stiffness tensor and the density at the boundary from $\Lambda_T^\Sigma$. We only need to recover from the surface impedance tensor $Z$ for the elliptic problem introduced above. We emphasize that in this subsection, we can take $T>0$ arbitrarily.

\begin{proposition} \label{VTI}
Assume that $\Sigma$ is flat.  For the following cases, the local DN map $\Lambda_T^\Sigma$ identifies  $(\mathbf{C},\rho)$ and all their
derivatives on $\Sigma$ uniquely. There is an explicit
reconstruction procedure for these identifications:
\begin{enumerate}
\item The stiffness tensor $\mathbf{C}$ is transversely isotropic in a neighborhood of $\Sigma$, with the symmetry axis normal to $\Sigma$;
\item The stiffness tensor $\mathbf{C}$ is orthorhombic in a neighborhood of $\Sigma$, with one of the three (known) symmetry planes tangential to $\Sigma$;
\end{enumerate}
\end{proposition}

The proof of the above proposition can be found in Appendix \ref{calculation}.
~\\

\begin{rem}
For the transversely isotropic case that the symmetric axis is nowhere normal to $\Sigma$, we can also obtain the reconstruction if $T$ is large enough. See Proposition $\ref{TTI}$ in the next section. 
\end{rem}
~\\

\begin{proposition} \label{homog}
Assume that the stiffness tensor $\mathbf{C}$ is homogeneous in a neighborhood of $\Sigma$,  and $\Sigma$ has a curved part. The local DN map $\Lambda_T^\Sigma$ identifies  $(\mathbf{C},\rho)$ in this neighborhood uniquely.
\end{proposition}
~\\

The proof of the  above proposition can be found in Appendix \ref{homo}.

\subsection{Recovery in the interior}${}$\newline\indent
We finally consider the recovery of a piecewise analytic density and
stiffness tensor in the interior of the domain. We begin with estimate (\ref{op_estimates}), leading to
\[\Lambda^{h,\Sigma}=\lim_{T\rightarrow\infty}h\mathcal{L}_T\Lambda_T^\Sigma\chi\chi_1^{-1}(h;T).\]
Therefore, we can consider the fully elliptic problem if we have
$\Lambda_{T^*}^\Sigma$, where $T^*>2r$ with $r$ defined in Lemma~\ref{Holmgren}
as the data, and adapt the procedure in \cite{CN} to study the problem
of recovering piecewise analytic material parameters. Once we have the boundary determination at $\Sigma$, by analyticity of the coefficients in subdomain $D_1$, we can propagate the data
to the interior interface $\Sigma_2$, and iterate the boundary determination results. We will sketch the procedure below in detail.

Now, we have the exact elliptic local DN map $\Lambda^{h,\Sigma}$. For the TI case, if the symmetry axis is normal to
$\Sigma$, we can recover the parameters on the boundary from the semiclassical symbol of $\Lambda^{h,\Sigma}$. However, if the symmetry axis is not normal to $\Sigma$, this
approach would fail.
Then we consider $\Lambda^{h,\Sigma}$ as a classical
pseudodifferential operator, and adapt the procedure developed in \cite{NTU} for their
reconstruction. (We can reduce to the above two situations by possibly
passing to further subset of $\Sigma$.)~\\

\begin{proposition}\label{TTI}
Assume that $\mathbf{C}$ is smooth and transversely isotropic with
symmetry axis nowhere normal to $\Sigma$, and assume that $\rho$ is
smooth. Then we have an explicit reconstruction of $\rho$ and
$\mathbf{C}$, as well as their derivatives on $\Sigma$, from the
symbols of $\Lambda^{h_1,\Sigma}$ and $\Lambda^{h_2,\Sigma}$, $h_1 \neq h_2$,
considered as classical pseudodifferential operators.
\end{proposition}
~\\

With all the boundary determination results developed before, we are ready to have the uniqueness for interior determination of piecewise analytic parameters.
We assume the domain partitioning introduced in Section \ref{cont} throughout this section. First, by Propositions \ref{VTI}, \ref{homog} and \ref{TTI}, we have

\begin{proposition}
If $\Lambda^{h,\Sigma}_{\mathbf{C}_1,\rho_1}=\Lambda^{h, \Sigma}_{\mathbf{C}_2,\rho_2}$ and on $D_1$, $\mathbf{C}_j,\rho_j$, $j=1,2$ are analytic and one of the following conditions holds:
\begin{enumerate}
\item $\mathbf{C}_j$ are TI with a known symmetry axis, that is, there exist Cartesian coordinates $y$ in $D_1$, such that the nonzero components of $\mathbf{C}_j(y)$ are those listed in Appendix \ref{calculation};
\item $\Sigma$ is flat, and $\mathbf{C}_j$ are orthorhombic with one
  of the three (known )symmetry planes tangential to $\Sigma$, that
  is, there exist Cartesian coordinates $y$ in $D_1$, such that the
  nonzero components of $\mathbf{C}_j(y)$ are those listed in
  Appendix \ref{calculation};
\item $\Sigma$ is partly curved, and $\mathbf{C}_j,\rho_j$, $j=1,2$ are constant in $D_1$;
\end{enumerate}
then $\mathbf{C}_1=\mathbf{C}_2$, $\rho_1=\rho_2$ on $D_1$.
\end{proposition}
~\\


In order to use the boundary determination results to have the uniqueness in the interior, we need the propagation of the DN map. Let $D_\beta,\beta=1,2\,\cdots,\alpha$ be a chain of subdomains of $\Omega$ such that $\Sigma_1:=\Sigma\subset\partial D_1$. Here the chain of subdomains $D_\beta$'s means that it satisfies the following conditions: (i) $D_\beta\cap D_{\beta'}=\emptyset$ if $\beta\neq\beta'$; (ii) there are nonempty smooth surfaces $\Sigma_\beta\subset\overline{D_\beta}\cap\overline{D_{\beta+1}}$, $\beta=1,2,\cdots,\alpha-1$. Further let $\Omega_\alpha=\Omega\setminus \cup_{\beta=1}^{\alpha-1} \overline{D_\beta}$, and $\Sigma_\alpha\subset\partial\Omega_\alpha$ be open, connected and smooth. Define $\Lambda^{h,\Sigma_\alpha}_{\mathbf{C},\rho}$ similar to  $\Lambda^{h,\Sigma}_{\mathbf{C},\rho}$ with $(\Omega,\Sigma)$ replaced by $(\Omega_\alpha,\Sigma_\alpha)$. By adapting the argument in \cite{Ike}, we have the following results analogous to \cite{CN}:


\begin{theorem}
Suppose that $\Lambda^{h,\Sigma}_{\mathbf{C}_1,\rho_1}=\Lambda^{h,\Sigma}_{\mathbf{C}_2,\rho_2}$. If on each subdomain $D_\alpha$, $\mathbf{C}_j,\rho_j$, $j=1,2$ are analytic up to the boundary of $D_\alpha$ and satisfy either of the following conditions on $D_\alpha$:
\begin{enumerate}
\item $\mathbf{C}_j$  is TI with a known symmetry axis;
\item $\Sigma_\alpha$ is flat and $\mathbf{C}_j$ are orthorhombic with one of the three symmetry planes tangential to $\Sigma_\alpha$ for each $\alpha$;
\item $\Sigma_\alpha$ is partly curved and $D_\alpha$, $\mathbf{C}_j, \rho_j$, $j=1,2$ are
constant;
\end{enumerate}
then $\mathbf{C}_1=\mathbf{C}_2$, $\rho_1=\rho_2$.
\end{theorem}
~\\


We introduce the notion of subanalytic set: $A\subset\mathbb{R}^3$ is
said to be subanalytic if for any $x\in\overline{A}$, there exists an
open neighorhood $U$ of $x$, real analytic compact manifolds
$Y_{i,j}$, $i=1,2$, $1\leq j\leq N$ and real analytic maps
$\Phi_{i,j}:Y_{i,j}\rightarrow \mathbb{R}^3$ such that 
\[
A\cap U=\cup_{j=1}^N(\Phi_{1,j}(Y_{1,j})\setminus\Phi_{2,j}(Y_{2,j}))\cap U.
\] 
For more details and nice properties about subanalytic sets, we refer to \cite{KV} and \cite{CN}. We note here that a polyhedron with a piecewise analytic boundary is a subanalytic set.  We also emphasize that the family of subanalytic sets is closed under finite union and finite intersection. Moreover, for two relatively compact subanlytic subsets $A$ and $B$, the number of connected components of $A\cap B$ is always finite. With this property, if we have two domain partitioning $\overline{\Omega}=\cup_{\alpha}\overline{D^{(1)}_\alpha}=\cup_{\beta}\overline{D^{(2)}_\beta}$ by two sets of subdomains $D_{\alpha}^{(1)}$'s and $D_{\beta}^{(2)}$'s such that each set of subdomains are mutually disjoint subanalytic sets, we can consider the finer domain partitioning
\[
\overline{\Omega}=\cup_{\gamma}\overline{\tilde{D}_\gamma},
\]
where each $\tilde{D}_\gamma$ is a connected component of $D_\alpha^{(1)}\cap D_\beta^{(2)}$ for some $\alpha$ and $\beta$.
Therefore, with the subanalytic property of the subdomains, we can recover the domain partitioning as well by adapting the argument of \cite{CN}. 

\begin{theorem}
Suppose that $\Lambda^{h,\Sigma}_{\mathbf{C}_1,\rho_1} =
\Lambda^{h,\Sigma}_{\mathbf{C}_2,\rho_2}$ and $\Sigma$ is curved. Let
on each subdomain $D_\alpha^{(j)}$, $\mathbf{C}_j, \rho_j$ be constant
for $j=1,2$. Let, furthermore, $\Omega$ and each $D_\alpha^{(j)}$ be
open subanalytic subsets of $\mathbb{R}^3$, and all the boundaries
$\partial D_\alpha^{(j)}\setminus\partial\Omega$ contain no open flat
subsets. Then $\mathbf{C}_1 = \mathbf{C}_2$ and $\rho_1 = \rho_2$.
\end{theorem}
~\\

Moreover, for isotropic elasticity, that is,
\begin{equation}\label{iso}
C^{ijkl}=\lambda(\delta^{ij}\delta^{kl})+\mu(\delta^{ik}\delta^{jl}+\delta^{il}\delta^{jk})
\end{equation}
we have
~\\

\begin{corollary}
Suppose that $\Lambda^{h,\Sigma}_{\mathbf{C}_1,\rho_1} =
\Lambda^{h,\Sigma}_{\mathbf{C}_2,\rho_2}$ and $\mathbf{C}$ is
isotropic of the form $(\ref{iso})$. Let on each subdomain
$D_\alpha^{(j)}$, $\lambda_j, \mu_j, \rho_j$ be analytic for
$j=1,2$. Let, furthermore, $\Omega$ and each $D_\alpha^{(j)}$ be open
subanalytic subsets of $\mathbb{R}^3$. Then $\lambda_1 = \lambda_2$,
$\mu_1 = \mu_2$ and $\rho_1 = \rho_2$.
\end{corollary}

The first and third authors, with collaborators, proved a uniqueness and Lipschitz stability for piecewise homogeneous isotropic elastic parameters $\lambda,\mu,\rho$ with time-harmonic DN map \cite{BdHFVZ}. From the uniqueness point of view, the above theorem is a more general result with nice enough properties of domain partitioning. 

\section*{Acknowledgment}
M. V. de Hoop gratefully acknowledges support from the Simons Foundation under the MATH + X program, the National Science Foundation under grant DMS-1559587, and the corporate members of the Geo-Mathematical Group at Rice University. G. Nakamura acknowledges the supports from Grant-in-Aid for Scientific Research (15K21766 and 15H05740) of the Japan Society for the Promotion of Science (JSPS).
\appendix

\renewcommand{\theequation}{\Alph{section}.\arabic{equation}}

\setcounter{equation}{0}
\section{Reconstruction of density and stiffness tensor at the boundary from the surface impedance tensor} \label{calculation}${}$\newline\indent
We present the reconstruction scheme for the material parameters (with certain symmetries) at the boundary from the surface impedance tensor $Z$ introduced above.

\subsection{Vertically transversely isotropic case}${}$\newline\indent
We first consider the vertically transversely isotropic case. We
assume that $\Sigma$ is flat and let the outer normal unit vector be
$\mathbf{n} = (0,0,1)$ with respect to the Cartesian coordinates
$y=(y^1,y^2,y^3)$; we assume that the axis of symmetry is aligned with
this normal. Then the nonvanishing components of the VTI stiffness
tensor, $\mathbf{C}$, are
\begin{equation}\label{TInonzero}
C^{1111},~C^{2222},~C^{3333},~C^{1122},~C^{1133},~C^{2233},~C^{2323},~C^{1313},~C^{1212}
\end{equation}
with relations
\[C^{1111}=C^{2222}, ~~C^{1133}=C^{2233},\]
\[C^{2323}=C^{1313},~~C^{1212}=\frac{1}{2}(C^{1111}-C^{1122}).\]
The strong convexity condition is equivalent to 
\[C^{1313}>0,~C^{1212}>0,~C^{3333}>0,~(C^{1111}+C^{1122})C^{3333}>2(C^{1133})^2.\]
In a neighborhood of $\Sigma$, we can use boundary normal coordinates $x$ and the Cartesian coordinates $y$ identically. Using the notation in Section \ref{symbol} and suppressing the dependence on $\mathbf{n}$, $D,\,Q,\,R$ take the forms
\[D=\left(\begin{array}{ccc}
C^{1313} & 0&0\\
0&C^{1313}&0\\
0&0&C^{3333}
\end{array}\right),~~~~~R=\left(\begin{array}{ccc}
0&0 &C^{1133}m_1\\
0&0&C^{1133}m_2\\
C^{1313}m_1&C^{1313}m_2&0
\end{array}\right),\]

\[Q=\left(\begin{array}{ccc}
C^{1111}m_1^2+C^{1212}m_2^2& (C^{1212}+C^{1122})m_1m_2&0\\
(C^{1212}+C^{1122})m_1m_2&C^{1212}m_1^2+C^{1111}m_2^2&0\\
0&0&C^{1313}(m_1^2+m_2^2)
\end{array}\right).\]
Writing 
\[P(\mathbf{m})=\left(\begin{array}{ccc}
|\mathbf{m}|^{-1}m_2&|\mathbf{m}|^{-1}m_1&0\\
-|\mathbf{m}|^{-1}m_1&|\mathbf{m}|^{-1}m_2&0 \\
0 & 0 & 1
\end{array}\right),\]
we find that
\[\hat{D}=P(\mathbf{m})^*DP(\mathbf{m})=\left(\begin{array}{ccc}
C^{1313} & 0&0\\
0&C^{1313}&0\\
0&0&C^{3333}
\end{array}\right),\]

\[
\hat{R}(\mathbf{m})=P(\mathbf{m})^*R(\mathbf{m})P(\mathbf{m})=\left(\begin{array}{ccc}
0&0 &0\\
0&0&C^{1133}|\mathbf{m}|\\
0&C^{1313}|\mathbf{m}|&0
\end{array}\right),
\]

\[
\hat{Q}(\mathbf{m})=P(\mathbf{m})^*Q(\mathbf{m})P(\mathbf{m})=\left(\begin{array}{ccc}
C^{1212}|\mathbf{m}|^2& 0&0\\
0&C^{1111}|\mathbf{m}|^2&0\\
0&0&C^{1313}|\mathbf{m}|^2
\end{array}\right).
\]

We emphasize the block-diagonal structure of the above matrices, and our later
calculations will rely on this. We note that $P(\mathbf{m})$ acts as a
block-diagonalizer of $D,R,Q$ in the above calculation. Without this
block diagonalization, the calculation of $Z(\mathbf{m})$ would not be
possible.
~\\

\begin{rem}
We note that the block diagonal structure is closely related to the
decoupling of surface wave modes. The 1-by-1 block corresponds to Love
waves and the 2-by-2 block corresponds to Rayleigh waves. We refer to
\cite{dHINZ} for further discussions.
\end{rem}
~\\

Exploiting the commutativity, $D P(\mathbf{m}) = P(\mathbf{m}) D$, we obtain the decomposition
\[S_0(\mathbf{m})=P(\mathbf{m})D^{-1/2}(A+\mathrm{i}B)D^{1/2}P(\mathbf{m})^*,\]
where
\begin{equation}\label{ABform}
   A=\left(\begin{array}{ccc}
0 & 0 &0\\
0 & 0 &  -\alpha_1\\
0 &  -\alpha_2 & 0
\end{array}\right) ,\quad
B=\left(\begin{array}{ccc}
a & 0 & 0\\
0 & b & 0\\
0 & 0 & c
\end{array}\right)
\end{equation}
with
\[\alpha_1=\frac{1}{1+\gamma}\frac{(C^{1133}+C^{1313})|\mathbf{m}|}{\sqrt{C^{1313}C^{3333}}},~~\alpha_2=\gamma\alpha_1,\]
\[c=\sqrt{\frac{C^{1313}|\mathbf{m}|^2+\rho}{C^{3333}}-\frac{(C^{1133}+C^{1313})^2|\mathbf{m}|^2}{(1+\gamma)^2C^{1313}C^{3333}}},~~b=\gamma c,\]
\[a=\sqrt{\frac{C^{1212}|\mathbf{m}|^2+\rho}{C^{1313}}},~~\gamma=\sqrt{\frac{(C^{1111}|\mathbf{m}|^2+\rho)C^{3333}}{(C^{1313}|\mathbf{m}|^2+\rho)C^{1313}}}.\]
Then
\begin{multline}\label{symbol_VTI}
   Z(\mathbf{m}) = -\mathrm{i}(DS_0+R^T)
   = P(\mathbf{m})
\\
   \left(\begin{array}{ccc}
C^{1313}a & 0 & 0\\
0 & C^{1313}b& \mathrm{i}\sqrt{C^{1313}C^{3333}}\alpha_1-\mathrm{i}C^{1313}|\mathbf{m}|\\
0 & \mathrm{i}\sqrt{C^{1313}C^{3333}}\alpha_2-\mathrm{i}C^{1133}|\mathbf{m}|& C^{3333}c
\end{array}
\right)P(\mathbf{m})^* .
\end{multline}

\subsection{Orthorhombic case}${}$\newline\indent
We assume, as before, that $\Sigma$ is flat and let the outer normal
unit vector be $\mathbf{n} = (0,0,1)$ with respect to the Cartesian
coordinates $y=(y^1,y^2,y^3)$; we assume that the coordinate axes span
the symmetry planes. Then the nonvanishing components of $\mathbf{C}$ are
\begin{equation}\label{orthorhombic}
C^{1111},~C^{2222},~C^{3333},~C^{1122},~C^{1133},~C^{2233},~C^{2323},~C^{1313},~C^{1212}.
\end{equation}
The matrices $D, Q, R$ take the form
\[D=\left(\begin{array}{ccc}
C^{1313} & 0&0\\
0&C^{2323}&0\\
0&0&C^{3333}
\end{array}\right),~~~~R=\left(\begin{array}{ccc}
0&0 &C^{1133}m_1\\
0&0&C^{2233}m_2\\
C^{1313}m_1&C^{2323}m_2&0
\end{array}\right),\]

\[Q=\left(\begin{array}{ccc}
C^{1111}m_1^2+C^{1212}m_2^2& (C^{1212}+C^{1122})m_1m_2&0\\
(C^{1212}+C^{1122})m_1m_2&C^{1212}m_1^2+C^{2222}m_2^2&0\\
0&0&C^{1313}m_1^2+C^{2323}m_2^2
\end{array}\right).\]
We have block-diagonalizing matrix as for the VTI case. For particular
directions of $\mathbf{m}$, we have the following block diagonal
structure. For $\mathbf{m}=(0,|\mathbf{m}|)$, we find that
\[R(|\mathbf{m}|\mathbf{e}_2)=\left(\begin{array}{ccc}
0&0 &0\\
0&0&C^{2233}|\mathbf{m}|\\
0&C^{2323}|\mathbf{m}|&0
\end{array}\right),\]

\[Q(|\mathbf{m}|\mathbf{e}_2)=\left(\begin{array}{ccc}
C^{1212}|\mathbf{m}|^2& 0&0\\
0&C^{2222}|\mathbf{m}|^2&0\\
0&0&C^{2323}|\mathbf{m}|^2
\end{array}\right)\]
so that
\begin{equation}\label{symbol_O1}
Z(|\mathbf{m}|\mathbf{e}_2)=\left(
\begin{array}{ccc}
C^{1313}a^1 & 0 & 0\\
0 & C^{2323}b^1& \mathrm{i}\sqrt{C^{2323}C^{3333}}\alpha^1_1-\mathrm{i}C^{2323}|\mathbf{m}|\\
0 & \mathrm{i}\sqrt{C^{2323}C^{3333}}\alpha^1_2-\mathrm{i}C^{2233}|\mathbf{m}|& C^{3333}c^1,\end{array}
\right),
\end{equation}
where
\[\alpha^1_1=\frac{1}{1+\gamma^1}\frac{(C^{2233}+C^{2323})|\mathbf{m}|}{\sqrt{C^{2323}C^{3333}}},~~\alpha^1_2=\gamma^1\alpha^1_1,\]
\[c^1=\sqrt{\frac{C^{2323}|\mathbf{m}|^2+\rho}{C^{3333}}-\frac{(C^{2233}+C^{2323})^2|\mathbf{m}|^2}{(1+\gamma^1)^2C^{2323}C^{3333}}},~~b^1=\gamma^1 c^1,\]
\[a^1=\sqrt{\frac{C^{1212}|\mathbf{m}|^2+\rho}{C^{1313}}},~~\gamma^1=\sqrt{\frac{(C^{2222}|\mathbf{m}|^2+\rho)C^{3333}}{(C^{2323}|\mathbf{m}|^2+\rho)C^{2323}}}.\]

Similarly, for $\mathbf{m}=(|\mathbf{m}|,0)$, we obtain
\begin{equation}\label{symbol_O2}Z(|\mathbf{m}|\mathbf{e}_1)=\left(
\begin{array}{ccc}
C^{1313}b^2 & 0 & \mathrm{i}\sqrt{C^{1313}C^{3333}}\alpha^2_1-\mathrm{i}C^{1313}|\mathbf{m}|\\
0 & C^{2323}a^2& 0\\
\mathrm{i}\sqrt{C^{1313}C^{3333}}\alpha^2_2-\mathrm{i}C^{1133}|\mathbf{m}| & 0 & C^{3333}c^2
\end{array}
\right),
\end{equation}
with
\[\alpha^2_1=\frac{1}{1+\gamma^2}\frac{(C^{1133}+C^{1313})|\mathbf{m}|}{\sqrt{C^{1313}C^{3333}}},~~\alpha^2_2=\gamma^2\alpha^2_1,\]
\[c^2=\sqrt{\frac{C^{1313}|\mathbf{m}|^2+\rho}{C^{3333}}-\frac{(C^{1133}+C^{1313})^2|\mathbf{m}|^2}{(1+\gamma^2)^2C^{1313}C^{3333}}},~~b^2=\gamma^2 c^2,\]
\[a^2=\sqrt{\frac{C^{1212}|\mathbf{m}|^2+\rho}{C^{2323}}},~~\gamma^2=\sqrt{\frac{(C^{1111}|\mathbf{m}|^2+\rho)C^{3333}}{(C^{1313}|\mathbf{m}|^2+\rho)C^{1313}}}.\]

\subsection{The reconstruction scheme}\label{rec}${}$\newline\indent
In this section, we give a reconstruction scheme for the material
parameters. The VTI and orthorhombic cases have the same structure: We
only need to consider the $Z^{(11)}(|\mathbf{m}|\mathbf{e}_1)$ element
in (\ref{symbol_O1}) and the $Z^{(11)}(|\mathbf{m}|\mathbf{e}_2)),
Z^{(13)}(|\mathbf{m}|\mathbf{e}_2),
Z^{(31)}(|\mathbf{m}|\mathbf{e}_2),$
$Z^{(33)}(|\mathbf{m}|\mathbf{e}_2)$ elements in
(\ref{symbol_O2}). Indeed, we only need to consider the VTI case.

We first make some basic algebraic observations. We note that $Z$ is a
Hermitian matrix, and contains 4 nonzero elements for the VTI
case. However, we have a total number of 6 unknowns to recover. This
is feasible, because these elements are non-homogeneous in
$\mathbf{m}$. Hence, different values for $|\mathbf{m}|$ give
different information. It also becomes clear why the VTI or
orthorhombic elastic parameters cannot be recovered from elastostatic
data \cite{NT,NTU}. In the VTI case, the surface impedance tensor for
elastostatics is homogeneous in $\mathbf{m}$, and thus we can only have 4
equations for 5 parameters. Then, it is impossible to recover
all the parameters.

We begin with a basic

\begin{lemma}\label{l1}
Consider a rational function 
\[f(t)=a+\frac{c}{t+b}\]
defined on $(0,\infty)$, then $a,b,c$ can be recovered from the values of $f(1),f'(1),f''(1)$.
\end{lemma}
\begin{proof}
It is immediate that
\[f'(1)=-\frac{c}{(1+b)^2},~~~~f''(1)=\frac{2c}{(1+b)^3},\]
from which we recover
\[b=\frac{2f'(1)}{f''(1)}-1.\]
Then we recover
\[a=f(1)+f'(1)(1+b),\]
and then
\[c=(f(1)-a)(1+b),\]
completing the reconstruction.
\end{proof}

\noindent\textit{Step 1.} From $(Z^{(11)}(|\mathbf{m}|\mathbf{e}_2)^2=C^{1212}C^{1313}|\mathbf{m}|^2+\rho C^{1313}$. By taking the difference with $|\mathbf{m}|=1$ and $|\mathbf{m}|=\sqrt{2}$, we recover
\[C^{1212}C^{1313}=(Z^{(11)}(\sqrt{2}\mathbf{e}_2)^2-(Z^{(11)}(\mathbf{e}_2))^2\]
and
\[\rho C^{1313}=(Z^{(11)}(\mathbf{e}_2))^2-C^{1212}C^{1313}.\]

\noindent\textit{Step 2.}  From $Z^{(22)}(|\mathbf{m}|\mathbf{e}_2))$ and $Z^{(33)}(|\mathbf{m}|\mathbf{e}_2))$, we recover
\[d_1(|\mathbf{m}|^2):=\frac{C^{1313}}{C^{3333}}\frac{C^{1111}|\mathbf{m}|^2+\rho}{C^{1313}|\mathbf{m}|^2+\rho}=\frac{(Z^{(22)}(|\mathbf{m}|\mathbf{e}_2))^2}{(Z^{(33)}(|\mathbf{m}|\mathbf{e}_2))^2}.\]
Viewed as a rational function defined on $(0,\infty)$,
\[d(t)=d_1(t^2)=\frac{C^{1111}}{C^{3333}}+\frac{(\rho C^{1313}-\rho C^{1111})/(C^{1313}C^{3333})}{t+\frac{\rho}{C^{1313}}}\]
and applying Lemma \ref{l1}, from the values of $d(1),d'(1),d''(1)$ we recover
\[\frac{C^{1111}}{C^{3333}},~~~~ \frac{\rho}{C^{1313}~}~~\text{and}~~~\frac{\rho C^{1313}-\rho C^{1111}}{C^{1313}C^{3333}}.\]
The recovery of
\[\frac{\rho}{ C^{3333}}= \frac{\rho C^{1313}-\rho C^{1111}}{C^{1313}C^{3333}}+\frac{\rho}{C^{1313}}\frac{C^{1111}}{C^{3333}}\]
follows immediately.

\noindent\textit{Step 3.} With these recoveries, we successively
obtain $\rho, C^{1313}, C^{1212}, C^{3333}$ and then $C^{1111}$.

\noindent\textit{Step 4.}  Now we have sufficient information to recover
 \[\gamma = \sqrt{\frac{(C^{1111}|\mathbf{m}|^2+\rho)C^{3333}}{(C^{1313}|\mathbf{m}|^2+\rho)C^{1313}}}\] 
 as a function of $|\mathbf{m}|$. We note that
\[\frac{C^{1313}}{C^{3333}}(Z^{(33)}(|\mathbf{m}|\mathbf{e}_2))^2=(C^{1313})^2|\mathbf{m}|^2+\rho C^{1313}-\frac{(C^{1133}+C^{1313})^2|\mathbf{m}|^2}{(1+\gamma)^2}.\]
Thus we recover
\[C^{1133}=\left(1+\gamma(1)\right)\sqrt{-\frac{C^{1313}}{C^{3333}}(Z^{(33)}(\mathbf{e}_2))^2+(C^{1313})^2+\rho C^{1313}}-C^{1313}.\]

\noindent\textit{Step 5.} We proceed with recovering the partial derivatives of the material parameters. From the the above procedure and Lemma~\ref{l1}, we obtain
\[\partial\left(\frac{\rho}{C^{1313}}\right)=2\partial\left(\frac{d'(1)}{d''(1)}\right),\]
and
\[\partial(\rho C^{1313})=\partial\left(2(Z^{(11)}(\mathbf{e}_2))^2-(Z^{(11)}(\sqrt{2}\mathbf{e}_2))^2\right),\]
where $\partial$ stands for any $\partial_{y^j},\,j=1,2,3$. It follows that
\begin{eqnarray*}
C^{1313}\partial\rho-\rho\partial C^{1313} &=& 2\partial\left(\frac{d'(1)}{d''(1)}\right) (C^{1313})^2,
\\
C^{1313}\partial\rho+\rho\partial C^{1313} &=& \partial\left(2(Z^{(11)}(\mathbf{e}_2))^2-(Z^{(11)}(\sqrt{2}\mathbf{e}_2))^2\right).
\end{eqnarray*}
Solving the $2$-by-$2$ linear system, we recover $\partial\rho$ and $\partial C^{1313}$.

\noindent\textit{Step 6.} From the relation
\[C^{1212}\partial C^{1313}+C^{1313}\partial C^{1212} =\partial\left((Z^{(11)}(\sqrt{2}\mathbf{e}_2))^2-(Z^{(11)}(\mathbf{e}_2))^2\right),\]
we recover
\[\partial C^{1212}=\frac{1}{C^{1313}}\left(C^{1212}\partial C^{1313}-\partial\left((Z^{(11)}(\sqrt{2}\mathbf{e}_2))^2-(Z^{(11)}(\mathbf{e}_2))^2\right)\right)\]

\noindent\textit{Step 7.} Again, following Lemma~\ref{l1},
\[\partial\left(\frac{C^{1111}}{C^{3333}}\right)=\partial\left(d(1)+d'(1)\left(1+\frac{\rho}{C^{1313}}\right)\right).\]

\noindent\textit{Step 8.} We note that
\[\partial\left(\frac{\rho C^{1313}-\rho C^{1111}}{C^{1313}C^{3333}}\right)=\partial\left(\left(d(1)-\frac{C^{1111}}{C^{3333}}\right)\left(1+\frac{\rho}{C^{1313}}\right)\right),\]
whence
\[C^{3333}\partial\rho-\rho \partial C^{3333}=(C^{3333})^2\partial\left(\left(d(1)-\frac{C^{1111}}{C^{3333}}\right)\left(1+\frac{\rho}{C^{1313}}\right)+\frac{\rho}{C^{1313}}\frac{C^{1111}}{C^{3333}}\right).\]
We recover $ \partial C^{3333}$. Furthermore, we note that
\[C^{3333}\partial C^{1111}-C^{1111}\partial C^{3333}=(C^{3333})^2\partial\left(d(1)+d'(1)\left(1+\frac{\rho}{C^{1313}}\right)\right),\]
from which we recover $\partial C^{1111}$.

\noindent\textit{Step 9.} We recover
\[\partial C^{1133}=\partial\left(\left(1+\gamma(1)\right)\sqrt{-\frac{C^{1313}}{C^{3333}}(Z^{(33)}(\mathbf{e}_2))^2+(C^{1313})^2+\rho C^{1313}}-C^{1313}\right).\]

\noindent\textit{Step 10.} We recover higher-order derivatives
\[\partial^m\rho,\ \partial^m C^{1111},\ \partial^m C^{1313},\ \partial^m C^{3333},\ \partial^m C^{1133},\ \partial^m C^{1212}\]
from $\partial^m Z$ for $m=2,3,\cdots$, where $\partial^m$ stands for any $\partial_y^\alpha$ with $|\alpha|=m$. The procedures are similar to those for the recovery of the first-order derivatives.

\section{Reconstruction of constant density and stiffness tensor and
         interface curvature condition}\label{homo}${}$\newline\indent
In this section, we assume that the density and stiffness tensor are
constant.  We revisit operator $\mathcal{M}$, define in $(\ref{transformed eq})$,
\[M(\mathbf{m}+q\mathbf{n})=D(\mathbf{n})q^2+(R(\mathbf{m},\mathbf{n})+R(\mathbf{m},\mathbf{n})^T)q+Q(\mathbf{m})+\rho I\]
where $D, R, Q$ are all constant in $y$ and $(\mathbf{m},q)$ is
interpreted as the Fourier dual $y$.

We develop the relation between the surface impedance tensor and the
fundamental solution $\Gamma(y)$ of $\mathcal{M}$, satisfying
\[
\mathcal{M}\Gamma(y)=\delta(y).
\]
 The semiclassical Fourier transform $\hat{\Gamma}(\eta)$ of $\Gamma$,
 \[\hat{\Gamma}(\eta)=\int_{\mathbb{R}^3}e^{-\frac{\mathrm{i}y\cdot\eta}{h}}{}\Gamma(y)\mathrm{d}y\]
satisfies
\[M(\mathbf{m}+q\mathbf{n})\hat{\Gamma}(\mathbf{m}+q\mathbf{n})=I.\]
In this expression,
\begin{equation}\label{decomp}\begin{split}
M(\mathbf{m}+q\mathbf{n})&=(q-S_0(\mathbf{m},\mathbf{n})^*)D(q-S_0(\mathbf{m},\mathbf{n}))\\
&=(q-\overline{S_0(\mathbf{m},\mathbf{n})}^*)D(q-\overline{S_0(\mathbf{m},\mathbf{n})}),
\end{split}\end{equation}
where we use that $S_0^*D-DS_0=R+R^T=\overline{R+R^T}=\overline{S_0}^*D-D\overline{S_0}$, $S_0^*DS_0=Q+\rho I=\overline{S_0}^*T\overline{S_0}$. Here, again, $\text{Spec}(S_0)\subset\mathbb{C}^+$, $\text{Spec}(\overline{S_0})\subset\mathbb{C}^-$.

We denote the semiclassical inverse Fourier transform $\mathcal{F}_{q\rightarrow\sigma,h}^{-1}$ of $\hat{\Gamma}(\mathbf{m}+q\mathbf{n})$ by
\[\mathcal{F}_{q\rightarrow\sigma,h}^{-1}\hat{\Gamma}(\sigma,\mathbf{m},\mathbf{n})=\frac{1}{2\pi h}\int_{-\infty}^{+\infty} e^{\frac{\mathrm{i}q\sigma}{h}}\hat{\Gamma}(\mathbf{m}+q\mathbf{n})\mathrm{d}q.\]
Then
\begin{eqnarray*}
D(hD_\sigma-S_0)\mathcal{F}_{q\rightarrow\sigma,h}^{-1}\hat{\Gamma}(\sigma,\mathbf{m},\mathbf{n}) &=& 0~~~\text{for}~\sigma>0 ,
\\
D(hD_\sigma-\overline{S_0})\mathcal{F}_{q\rightarrow\sigma,h}^{-1}\hat{\Gamma}(\sigma,\mathbf{m},\mathbf{n}) &=& 0~~~\text{for}~\sigma<0 ,
\end{eqnarray*}
while
\[\lim_{\sigma\rightarrow 0+}hD_\sigma\mathcal{F}^{-1}_{q\rightarrow\sigma,h}\hat{\Gamma}(\sigma,\mathbf{m},\mathbf{n})-\lim_{\sigma\rightarrow 0-}hD_\sigma\mathcal{F}^{-1}_{q\rightarrow\sigma,h}\hat{\Gamma}(\sigma,\mathbf{m},\mathbf{n})=I.\]
From the fact that $\mathcal{F}^{-1}_{q\rightarrow\sigma,h}\hat{\Gamma}(\sigma,\mathbf{m},\mathbf{n})$ is continuous in the variable $\sigma$, we conclude that
\[\begin{split}
D(S_0-\overline{S_0})\mathcal{F}^{-1}_{q\rightarrow\sigma,h}\hat{\Gamma}(0,\mathbf{m},\mathbf{n})
= I.\end{split}\]
%
%
%
%
%
%
Recalling that
\[Z(\mathbf{m},\mathbf{n})=-\mathrm{i}(DS_0(\mathbf{m},\mathbf{n})+R^T(\mathbf{m},\mathbf{n})),\]
we find that
\[-\mathrm{i}D(S_0(\mathbf{m},\mathbf{n})-\overline{S_0}(\mathbf{m},\mathbf{n}))=2D\mathrm{Im} \{S_0(\mathbf{m},\mathbf{n})\}=2\mathrm{Re} \{Z(\mathbf{m},\mathbf{n})\}.\]
Therefore,
 \[\mathcal{F}^{-1}_{q\rightarrow\sigma,h}\hat{\Gamma}(0,\mathbf{m},\mathbf{n})=\frac{1}{2}\left(\text{Re}\{Z(\mathbf{m},\mathbf{n})\}\right)^{-1}.\]

We identify
\begin{equation}X\hat{\Gamma}(\mathbf{m},\mathbf{n})=\int_{\mathbb{R}}\hat{\Gamma}(\mathbf{m}+s\mathbf{n})\mathrm{d}s=2\pi h\mathcal{F}^{-1}_{q\rightarrow\sigma,h}\hat{\Gamma}(0,\mathbf{m},\mathbf{n}).\end{equation}
as the X-ray transform of $\hat{\Gamma}$ along $\mathbf{n}$. Thus
\[\Gamma(y)=\int e^{-\frac{\mathrm{i}y\cdot\mathbf{m}}{h}}X\hat{\Gamma}(\mathbf{m},\mathbf{n})\mathrm{d}\mathbf{m},\]
for any $y\perp\mathbf{n}$. We say $\Sigma$ is curved, if $\Sigma$ is locally represented by the graph of a function $y^3=\varphi(y^1,y^2)$ such that $D^2\varphi$ does not vanish. If $\Sigma$ is curved, 
we know $Z(\mathbf{m},\mathbf{n})$ for $\mathbf{n}$ in a continuous curve, joining two different points,
on $S^2$. Then we know $\Gamma(y)$ for $y$ in an open subset of
$\mathbb{R}^3\setminus\{0\}$. Since $\Gamma(y)$ is analytic in
$\mathbb{R}^3\setminus\{0\}$, we can recover $\Gamma(y)$, and thus
$\hat{\Gamma}(\eta)$ for all $\eta\in\mathbb{R}^3$. Following
\cite{CN} we then complete the reconstruction of the stiffness tensor
$\mathbf{C}$ and the density $\rho$. The curved boundary condition is first introduced in \cite{AdHG} for the recovery of a piecewise homogeneous, fully anisotropic conductivity.

\bibliographystyle{abbrv}

\end{document}